\documentclass[letterpaper, 10 pt, conference]{ieeeconf}  

\IEEEoverridecommandlockouts                              

\usepackage{amsthm}
\usepackage{amsmath}
\usepackage{amssymb}
\usepackage{color}
\usepackage{graphicx}

\usepackage[linesnumbered,lined,ruled,commentsnumbered]{algorithm2e}
\RestyleAlgo{ruled}
\SetKwComment{Comment}{// }{}

\DeclareMathOperator*{\argmax}{arg\,max}

\newtheorem{theorem}{Theorem}
	\newtheorem{prop}[theorem]{Proposition}
	\newtheorem{lemma}[theorem]{Lemma}

	\theoremstyle{definition}
	\newtheorem{definition}[theorem]{Definition}
	\newtheorem{assumption}[theorem]{Assumption}
	\newtheorem{remark}[theorem]{Remark}

	\newcommand{\R}{\mathbb{R}}
	
	\newcommand{\U}{\mathcal{U}}

    \newcommand{\KK}{\mathbf{K}}

	\newcommand{\A}{\mathcal{A}}

	\newcommand{\eps}{\varepsilon}

\makeatletter
\let\NAT@parse\undefined
\makeatother
\usepackage[linkcolor=blue]{hyperref}
\hypersetup{
	colorlinks   = true, 
	citecolor   = red 
}

\title{\LARGE \bf
A minimax optimal control approach for robust neural ODEs
}

\author{Cristina Cipriani$^{1}$, Alessandro Scagliotti$^{2}$ and Tobias W\"ohrer$^{3}$
\thanks{$^{1,2,3}$Technical University Munich (TUM), Department of Mathematics, Munich, Germany.}
\thanks{$^{1}$Munich Data Science Institute (MDSI), Munich, Germany}
\thanks{$^{1,2}$Munich Center for Machine Learning (MCML), Munich, Germany
\textit{cristina.cipriani@ma.tum.de, scag@ma.tum.de, tobias.woehrer@tum.de}}
}

\begin{document}

\maketitle
\thispagestyle{empty}
\pagestyle{empty}

\begin{abstract}
In this paper, we address the adversarial training of neural ODEs from a robust control perspective.
This is an alternative to the classical training via empirical risk minimization, and it is widely used to enforce reliable outcomes for input perturbations.
Neural ODEs allow the interpretation of deep neural networks as discretizations of control systems, unlocking powerful tools from control theory for the development and the understanding of machine learning.
In this specific case, we formulate the adversarial training with perturbed data as a minimax optimal control problem, for which we derive first order optimality conditions in the form of Pontryagin's Maximum Principle. We provide a novel interpretation of robust training leading to an alternative weighted technique, which we test on a low-dimensional classification task.
\end{abstract}

\section{INTRODUCTION}
The recent advances in deep neural networks (DNNs) combined with the rapidly increasing supply of computing power have lead to numerous breakthroughs in science and technology \cite{hassabis2021}. DNN algorithms have shown to possess remarkable data generalization properties and flexible application settings \cite{resnet2016, Good2016}. Nonetheless, it has become evident that many machine learning (ML) models suffer from a severe lack of robustness against data manipulation. \emph{Adversarial attacks} describe inputs with minuscule changes, e.g.\ pictures with added pixel noise that is invisible in low resolution, which result in dramatic changes in the model outputs \cite{szegedy2013, Good2015}. Such attack vulnerabilities limit the implementations of ML algorithms in areas that demand high reliability, such as self-driving vehicles or automated content filtering. The robustness crisis further highlights the general lack of theoretical understanding and interpretability of neural network algorithms.

A central theoretical achievement of recent years is the interpretation of DNNs, such as residual NNs, as dynamical systems \cite{weinan2017, haber2017}: The input-to-output information flow through a network with an infinite amount of layers can be formulated in the continuum limit. This leads to nonlinear \emph{neural ODEs} (nODE) \cite{CRBD18, dupont2019}, where time takes the role of the continuous-depth variable. The nODE vantage point has proven to be a powerful tool that allows to interpret the learning problems (or parameter training) of DNNs as continuous-time control problems and to formulate \emph{Pontryagin's Maximum Principle} (PMP) \cite{pont62} in various network settings \cite{LCCE19, CRBD18, doubletrillos}. The link between the training of DNNs with batches of input and simultaneous control led to numerous additional theoretical insights into neural networks \cite{CFS_23, EGPZ, ruiz2023neural}.

In the context of nODEs, an adversarial attack in predefined norm $\|\cdot\|$ of \emph{adversarial budget} $\eps>0$ takes the form of a perturbed initial data point $\tilde{x}^0= x^0 + \alpha$ where $\|\alpha\|\leq \eps$. As a result, the \emph{adversarially robust learning problem} has been established \cite{madry2018, shaham18} as the minimax control problem
\begin{equation}\label{eq:roboptim}
	\min_u  \mathbb{E}_{(x^0,y)\sim\mu} \left[\max_{\|\alpha\|\leq \eps} \operatorname{Loss}(u, x^0 + \alpha, y) \right]
\end{equation}
where the initial data and labels $(x^0,y)$ are drawn from an underlying (and generally unknown) data distribution $\mu$ and where the function $\operatorname{Loss}(u,x^0,y)$ quantifies the prediction accuracy resulting from the model control $u$. While such saddle-point problems have a long history \cite{wald45}, their non-smooth, non-convex nature pose challenges in obtaining theoretical result. Additionally, the inner maximization in \eqref{eq:roboptim} is $x^0$ dependent and in applications often high-dimensional. Numerical implementations are therefore typically substituting the maximization by first-order guesses of worst-case adversarial attacks \cite{Good2015, madry2018, tramer2017transfer, frank2022existence}. 

\emph{Contributions:}
In this work, we interpret the forward flow of DNNs as nonlinear neural ODEs and take a control theoretical approach. We investigate the non-smooth and non-convex robust learning problem \eqref{eq:roboptim} by formulating the corresponding PMP for a finite ensemble of data points. We provide a proof of this classical result by means of separation of Boltyanski approximating cones. In spite of being a well-known result, the novelty consists in relating the solution of the resulting PMP to the one of another smooth and properly weighted control problem. Inspired by this interpretation, we develop a numerical method to address the robust learning problem, and we test it on a bidimensional classification task.
\section{Robust control model} \label{sec:model}
In this paper, we consider the simultaneous and robust control problem of a finite number of particles in $\R^d$, whose initial positions are subject to perturbations (also known as \emph{attacks} in the machine learning literature).
Namely, the control system driving the dynamics of a single particle is:
\begin{equation}\label{eq:nODE}
\begin{cases}
		\dot{x}(t) = F(x(t),u(t)), & \mbox{a.e. } t\in[0,T],\\
		x(0) = x^0,
\end{cases}
\end{equation}
where $x^0\in \R^d$ represents the unperturbed initial datum, $F:\R^d \times \R^m\to \R^d$ is a continuous function (with Lipschitz dependence in the first variable), and $u\in \U:= L^2([0,T],U)$ is an admissible control taking values in any compact set $U\subset \R^m$. 
Here, we consider a collection (or \emph{batch}) of $M$ initial data points $\{x^0_1, \ldots, x^0_M\}$ for the system \eqref{eq:nODE}, which are \emph{simultaneously} driven by the same admissible control $u\in \U$.
In addition, each of these Cauchy data points is affected by $N$  additive perturbations, i.e.,
\begin{equation*}
    \A_i^N := \{ \alpha_i^1, \ldots, \alpha^N_i\}, \quad \forall i=1, \ldots, M,
\end{equation*}
where $\alpha_i^j \in \R^d$ ($j=1,\ldots,N$) is the $j$-th perturbation of $x^0_i$. We denote by $x_i^j:[0,T]\to\R^d$ the solution of \eqref{eq:nODE} corresponding to the initial condition $x^j_i(0)= x_i^0 + \alpha_i^j$, while we refer to the collection of the $N$ perturbations of the $i-$th particle using $X_i(\cdot) := (x_i^1(\cdot), \ldots, x_i^N(\cdot))\in \R^{d\times N}$.
Finally, we employ the variable $X=(X_1,\ldots,X_M)\in (\R^{d\times N})^M$ to describe the states of the whole system, i.e., the ensemble of every perturbation of every particle.
We are interested in the minimization of the functional $J:\U\to \R$, defined as
\begin{equation}\label{eq:def_funct}
    \begin{split}
         &J(u):= \max_{\substack{j_1 = 1, \ldots, N \\ \cdots \\ j_M=1,\ldots,N} } \frac{1}{M} \sum_{i=1}^M g_i\big( x_i^{j_i}(T) \big) \\ 
        & \quad = \frac{1}{M} \sum_{i=1}^M \max_{j = 1, \ldots, N} g_i\big(x_i^j(T)\big) = \frac{1}{M}\sum_{i=1}^M \tilde{g}_i(X_i(T)),
    \end{split}
\end{equation}
where $x_i^j(T)$ is the terminal-time value of the solution of \eqref{eq:nODE} starting from $x_i^0+\alpha_i^j$ and steered by the control $u\in \U$, while, for every $i=1,\ldots,M$,  $\tilde g_i:\R^{d\times N}\to \R$ is the non-smooth function defined as $\tilde g_i(X_i) = \max(g_i(x_i^1),\ldots,g_i(x_i^N))$, and $g_i:\R^d\to\R$ is a smooth convex function. Notice that for each datum $x_i$, we must consider various perturbations $\alpha_i^j$, depending on the initial datum. This motivates the need to compute the maximum over different indices $j_i$, which depend on the datum $i$. However, in the second identity of \eqref{eq:def_funct}, we exploit the fact that each perturbation is independent from the others to swap the sum and the maximization. \\
In machine learning, the minimization of \eqref{eq:def_funct} can be rephrased as the empirical loss minimization of the robust training problem~\eqref{eq:roboptim}, where $g_i$ incorporates the $i$-th data label, which has been denoted with the variable $y$ in \eqref{eq:roboptim}.
\begin{definition} \label{def:sl_minimizer}
Given $u\in \U$ and the corresponding collection of trajectories $X:[0,T]\to(\R^{d\times N})^M$ that solve \eqref{eq:nODE} with the perturbed initial data, we call the couple $(u,X)$ an \emph{admissible process}.
Moreover an admissible process $(\bar u, \bar X)$ is a \emph{strong local minimizer} for \eqref{eq:def_funct} if there exists an $\varepsilon>0$ such that, for every other admissible process $(u,X)$ satisfying $\|X(\cdot)-\bar X(\cdot)\|_{C^0} \leq \varepsilon$, we have
\begin{equation*}
		\frac{1}{M}\sum_{i=1}^M \tilde{g}_i(\bar X_i(T)) \leq \frac{1}{M} \sum_{i=1}^M \tilde{g}_i(X_i(T)).
\end{equation*}     
\end{definition}
Moreover, we further introduce 
\begin{equation} \label{eq:ext_node}
    \dot{X}(t) = \mathcal{F}(X(t),u(t)),  \quad \mbox{a.e. }
    t\in [0,T],
\end{equation}
to denote the evolution of the whole ensemble of trajectories in $(\R^{d\times N})^M$, where the mapping $\mathcal{F}:(\R^{d\times N})^M\times \R^m\to (\R^{d\times N})^M$ should be understood as the application of $F(\cdot, u)$ component-wise. This allows us to conveniently define the Hamiltonian $H:(\R^{d\times N})^M\times ((\R^{d\times N})^M)^\star \times \R^m \to \R$ as 
\begin{equation*}
\begin{split}
	H(X, P, u) : &= P \cdot \mathcal{F}(X, u)
	= \sum_{i=1}^M \sum_{j=1}^N p_i^j \cdot F\big(x_i^j,u\big),
\end{split}
\end{equation*}
where $P=(P_1,\ldots,P_M)$, $P_i=(p_i^1,\ldots,p_i^N)$, and $p_i^j\in \R^d$ for every $i=1,\ldots,M$ and every $j=1,\ldots,N$.

\subsection{Pontryagin Maximum Principle}
Below, we specify the hypotheses needed to formulate the necessary conditions for strong local minimizers of \eqref{eq:def_funct}.
\begin{assumption} 
We assume that
\begin{itemize}
    \item[(A1)] the function $F: \R^d \times \R^m \to \R^d$ that defines the dynamics in \eqref{eq:nODE} is continuous, and it is $C^2$-regular in the first variable. Moreover, there exists $C>0$ such~that
    \begin{equation*}
        |F(x,\omega)-F(y,\omega)|_2 \leq C(|\omega|_2 + 1)|x-y|_2,
    \end{equation*}
    for every $x,y \in \R^d$ and for every $\omega\in \R^m$;
    \item[(A2)] for every $i=1,\ldots,M$, the function $g_i: \R^d \to \R$ related to the terminal-cost in \eqref{eq:def_funct} is convex, $C^1$-regular, bounded from below and never equal to $+\infty$.
\end{itemize}
\end{assumption}
From now on, when we consider a convex function $f:\R^d\to\R$, we assume that it never attains the value $+\infty$. 
Let us now present the main result which motivates our numerical approach presented in Section~\ref{sec:numerics}.
\begin{theorem}[PMP for minimax]\label{thm:pmpfinite}
Let $(\bar u, \bar X)$ be a strong local minimizer for the minimax optimal control problem related to \eqref{eq:def_funct}. Then, there exist 
coefficients $(\gamma_i^j)_{i=1,\ldots,M}^{j=1,\ldots,N}$ satisfying $\gamma_i^j \geq 0$ and $\sum_{j=1}^N \gamma_i^j = 1 \quad \forall i=1, \ldots M$,
and there exists a curve $P: [0,T] \to \left((\R^{d \times N})^M\right)^\star$ 
whose components solve the adjoint equations
\begin{equation} \label{eq:pmp_covect}
\begin{cases}
\dot p_i^j(t) = p_i^j(t)  \nabla_x F\big( \bar x_i^j(t), \bar u(t) \big), & \mbox{a.e. } t\in [0,T],\\
 p_i^j(T) = -\frac1M \gamma_i^j  \nabla_x g_i \big(\bar x_i^j(T)\big),
\end{cases}
\end{equation}
such that, for a.e. $t \in [0,T]$, it holds that
\begin{equation} \label{eq:max_cond_pmp}
    H( \bar X(t), P(t), \bar u(t)) = \max_{\omega \in U} H(\bar X (t), P(t), \omega) .
\end{equation}
Moreover, for every $i=1,\ldots,M$, if $\gamma_i^j >0$, then
\begin{equation*}
    j \in \argmax_{j=1, \ldots, N} g_i \big(x_i^j(T)\big).
\end{equation*}
\end{theorem}

\begin{remark}\label{rmk:abnormals}
    In Theorem~\ref{thm:pmpfinite} only normal extremals are considered, since the problem of minimizing the functional $J:\U\to\R$ defined in \eqref{eq:def_funct} does not admit abnormal extremals.
\end{remark}

\begin{remark}\label{rmk:unbounded_controls}
    Here we consider controls $u:[0,T]\to\R^m$ taking values in an arbitrary compact set $U\subset\R^m$. However, it is possible to remove this constraint, and to add in the definition \eqref{eq:def_funct} of the functional $J$ the penalization $\beta\| u\|_{L^2}^2$ on the energy of the control, where $\beta>0$ is a parameter tuning this regularization.
    In this case, Theorem~\ref{thm:pmpfinite} holds in the same form with a slightly different Hamiltonian, which has now to take into account the running cost.
\end{remark}

\begin{remark} \label{rmk:measure_formulation}
In \eqref{eq:pmp_covect} it is possible to assign the following alternative terminal-time Cauchy datum:
\begin{equation*}
    p_i^j(T) = - \nabla_x g_i\big(x_i^j(T)\big)
\end{equation*}
for every $i=1,\ldots,M$ and $j=1,\ldots,N$.
In this case, we introduce the probability measure $d\Gamma_i$ on the set $\A_i^N=\{ \alpha_i^1,\ldots,\alpha_i^N\}$, defined as $d\Gamma_i(\alpha_i^j)=\gamma_i^j$, where $(\gamma_i^j)_{i=1,\ldots,M}^{j=1,\ldots,N}$ are the coefficients prescribed by Theorem~\ref{thm:pmpfinite}.
Therefore, we can rephrase \eqref{eq:max_cond_pmp} as
\begin{equation*}
\begin{split}
    &\frac{1}{M} \sum_{i=1}^M
 \int_{\A_i^N} p_i^j(t) \cdot F\big(\bar x_i^j(t), \bar u(t)\big) \,d \Gamma_i(\alpha_i^j) \\
 &\quad = \max_{\omega\in U} \frac{1}{M} \sum_{i=1}^
M \int_{\A_i^N} p_i^j(t) \cdot F(\bar x_i^j(t), \omega) \, 
d \Gamma_i(\alpha_i^j).
\end{split}
\end{equation*}
This formulation is particularly suitable to generalize to the case of infinitely many perturbations, that has been investigated in \cite{article_vinter}.
Furthermore, the case of infinitely many particles with finitely many perturbations per particle can be treated with a mean-field analysis, as recently done in \cite{BCFH_23,CFS_23} in the smooth framework.
\end{remark}

\subsection{Interpretation of the PMP}
An important message conveyed by Theorem~\ref{thm:pmpfinite} is that \emph{any strong local minimizer of the minimax problem related to the functional $J$ in \eqref{eq:def_funct} is an extremal for another smooth and properly weighted optimal control problem}. 
Namely, let us consider the functional $J_\Gamma:\U\to\R$ defined as
\begin{equation}\label{eq:funct_smooth}
    J_\Gamma(u) = \frac1M \sum_{i=1}^M \sum_{j=1}^N \gamma_i^j 
    g_i\big( x_i^j(T) \big),
\end{equation}
where $x_i^j:[0,T]\to\R^d$ are the same as in \eqref{eq:def_funct}, and where $(\gamma_i^j)_{i=1,\ldots,M}^{j=1,\ldots,N}$ are precisely the non-negative coefficients prescribed by Theorem~\ref{thm:pmpfinite} for the strong local minimizer $(\bar u, \bar X)$ of $J$.
Then, it turns out that $(\bar u, \bar X)$ satisfies the corresponding PMP conditions for the optimal control problem associated to $J_\Gamma$, which is \emph{smooth} with respect to the terminal-states.
This fact is extremely interesting from a numerical perspective, since the minimization of a functional with smooth terminal cost can be, in general, more easily addressed.
On the other hand, the coefficients $(\gamma_i^j)_{i=1,\ldots,M}^{j=1,\ldots,N}$ are not explicitly available a priori, as they are themselves part of the task of finding a strong local minimizer of $J$.
In Section~\ref{sec:numerics}, we propose and numerically test a procedure to iteratively adjust these coefficients.

\section{Proof of the PMP} \label{sec:proof}
The proof of Theorem~\ref{thm:pmpfinite} is classical and can be deduced, e.g., from \cite[Theorem~6.4.1]{book_vinter}.
However, in that case, the proof is rather sophisticated since it is developed in a very general framework.
In this section, we propose a more direct proof based on an \emph{Abstract Maximum Principle}, for which we refer to the recent expository paper \cite{book_rampazzo}.

\subsection{Abstract Maximum Principle}
In order to prove Theorem \ref{thm:pmpfinite}, we first need to collect some well known results of convex and non-convex analysis. 
\begin{lemma}[Result from \cite{book_convex}, Lemma~2.1.1 (Part~D)] \label{lem:taylor}
Given a convex function $f: \R^n \to \R$ and $x \in \R^n$, then 
\begin{equation} \label{eq:def_support}
    f(x+h) = f(x) + \sigma_{\partial f(x)}(h) + o(|h|) \quad \text{as } h \to 0,
\end{equation}
where $\partial f(x)$ denotes the subdifferential of $f$ at $x$.
\end{lemma}
We recall that, given a convex and compact set $K\subset \R^n$, the \emph{support function} $\sigma_K:\R^n\to \R$ is the convex sublinear function:
\begin{equation*}
    \sigma_K(x) := \sup_{p\in K} \langle p, x \rangle.
\end{equation*}
For a thorough discussion on the properties of support functions, we recommend \cite[Part~C]{book_convex}.
\begin{lemma} \label{lem:max_rule}
Given $N$ convex functions $f_j: \R^d \to \R$, $j=1, \ldots, N$, let us define $\Phi:\R^{d\times N}\to \R$ as
$\Phi(\xi) = \max_j f_j(x^j)$ for every $\xi=(x^1,\ldots,x^N)$.
Then, for every $\xi \in \R^{d\times N}$ we have that
\begin{equation*}
\begin{split}
    \partial \Phi(\xi) \subset 
\big\{ \big(\gamma^1 \partial f_1(x^1), &\ldots, \gamma^N \partial f_N(x^N)  \big): \\
& \gamma \in \Gamma,  \gamma^j=0 \text{ if } f_j(x^j) < \Phi(x)
\big\},
\end{split}
\end{equation*}
where $\Gamma := \{\gamma = (\gamma^1, \ldots, \gamma^N) \in \R^N\!: \gamma^j \geq 0, \sum_j\gamma^j =1\}$.
\end{lemma}
\begin{proof}
See \cite[Corollary~4.3.2 (Part~D)]{book_convex}. 
\end{proof}

We also recall the definition of Boltyanski approximating cones from \cite[Definition~2.1]{book_rampazzo}.
\begin{definition}
Consider a subset $\mathcal{K} \subseteq \R^n$ and a point $y \in \mathcal{K}$, and let $\KK$ be a convex cone in $\R^n$.
We say that $\KK$ is a \textit{Boltyanski approximating cone to $\mathcal{K}$ at $y$} if
$\KK = L C$ where,
\begin{enumerate}
    \item[$(i)$] for an integer $m>0$, $C \subseteq \R^m$ is a convex cone,
    \item[$(ii)$] $L: \R^m \to \R^n$ is a linear mapping,
    \item[$(iii)$] there exists $\delta>0$ and a continuous map $\mathsf F:  C \cap B_\delta(0) \to \mathcal{K}$ such that
    \begin{equation}
        \mathsf F(c) = y + L c + o(|c|), \quad \text{as }c \to 0.
    \end{equation}
\end{enumerate}
\end{definition}
In the next proposition we provide the construction of a Boltyanski approximating cone to the epigraph of a non-smooth convex function.
This represents a slight generalization w.r.t. the framework of \cite{book_rampazzo}, where the cost function is assumed to be differentiable. 
Indeed, in that case, it is sufficient to consider the \emph{graph} of the cost, and to take the tangent hyperplane as a Boltyanski approximating cone. 

\begin{prop} \label{prop:cone_epi}
Let $\mathcal{R} \subseteq \R^n$ be a set and, for $y \in \mathcal{R}$, let $\mathbf{R}\subseteq \R^n$ be a Boltyanski approximating cone at $y$ to $\mathcal{R}$. 
Let $f:\R^n \to \R$ be a convex function, and let $\mathrm{epi}(f_{|\mathcal{R}}) := \{(x,\eta) \in \R^{n+1}: x\in \mathcal{R}, \eta\geq f(x)\}$. 
Then, a Boltyanski approximating cone at $(y, f(y))$ to $\mathrm{epi}(f_{|\mathcal{R}})$ is
\begin{equation*}
    \tilde{\mathbf{R}} := \{(v, \sigma_{\partial f(y)}(v) + \eta) : v \in \mathbf{R}, \eta \geq 0\}.
\end{equation*}
\end{prop}
\begin{proof}
Since $\mathbf{R}$ is an approximating cone at $y$ to $\mathcal{R}$, there exists a convex cone $C \subseteq \R^m$, a linear mapping $L_{\mathbf{R}}:\R^m \to \R^n$, and a continuous function $\mathsf F_{\mathbf{R}}:C \to \mathcal{R}$ such that 
\begin{equation*}
    \mathsf F_{\mathbf{R}}(c) = y + L_{\mathbf{R}}c + o_1(|c|) \mbox{ as } c\to0  \mbox{ , and } \mathbf{R} = L_{\mathbf{R}} C.
\end{equation*}
Thus, we set $\tilde{{C}}:= 
\{ \big(c, \sigma_{\partial f(y)}(L_{\mathbf{R}} c) + \eta\big):
c \in C, \eta \geq 0\}.
$
We observe that $\tilde{{C}}$ is a convex cone, since it is the epigraph of the support function $c \mapsto \sigma_{\partial f(y)}(L_{\mathbf{R}} c)$ restricted to $C$. 
Moreover, let us set $\tilde{L}: \R^{m+1} \to \R^{n+1}$ such that $\tilde{L}(c,\mu) := (L_{\mathbf{R}} c, \mu)$ where $(c, \mu) \in \tilde{C}$, and let us define the continuous function $\tilde{\mathsf F}:\tilde C\to \mathrm{epi}(f|_{\mathcal{R}})$ as:
\begin{equation*}
    \tilde{\mathsf F}(c, \mu) := \big( \mathsf F_{\mathbf{R}}(c), \mu- \sigma_{\partial f(y)}(L_{\mathbf{R}}c) + f(\mathsf F_{\mathbf{R}}(c)) \big).
\end{equation*}
We observe that
\begin{equation*}
\begin{split}
    f(\mathsf F_{\mathbf{R}}(c)) &= f\big(y + L_{\mathbf{R}}c + o_1(|c|)\big)\\& = f(y) + \sigma_{\partial f(y)}(L_{\mathbf{R}} c) + o_2(|c|),   
\end{split}
\end{equation*}
where we have used Lemma~\ref{lem:taylor} and the boundedness of $\partial f(y)$.
Hence, for every $(c,\mu)\in \tilde C$, we have
\begin{equation*}
\begin{split}
    \tilde{\mathsf F}(c,\mu) &= \big(y, f(y) \big) + \big(L_{\mathbf{R}} c, \mu \big) + \big(o_1(|c|),o_2(|c|)\big) \\
    &= \big(y, f(y)\big) + \tilde{L}(c, \mu) +o(|c|),
\end{split}
\end{equation*}
and we set 
$\tilde{\mathbf{R}} := 
\tilde L \tilde C
$. Recalling that, for every $(c,\mu)\in \tilde C$, $\tilde{L}(c, \mu) =
    \tilde{L}\big(c, \sigma_{\partial f(y)}(L_\mathbf{R} c) + \eta\big)  
    =\big( L_\mathbf{R} c, \sigma_{\partial f(y)}(L_\mathbf{R} c) + \eta\big)$,
we deduce the thesis.
\end{proof}
Finally, we report here an extension of the Abstract  Maximum Principle presented in \cite[Section~5]{book_rampazzo}, aimed at encompassing the case of a non-smooth convex cost.
\begin{theorem} \label{thm:abstr_PMP}
Let us consider  a metric space $\U$,  a continuous map $y:\U\to\R^n$, a convex function $\Psi:\R^n\to\R$, and a set $\mathcal{S}\subset \R^n$. Let $(\bar u, y(\bar u))\in \U\times \R^n$ be a \emph{strong local minimizer}\footnote{This concept is the natural generalization of Definition~\ref{def:sl_minimizer}. See also \cite[Definition~5.1 and Remark~5.2]{book_rampazzo}.} for 
\begin{equation*}
    \min_{u \in \mathcal{U}} \Psi(y(u)) \quad \mbox{subject to } y(u) \in \mathcal{S}. 
\end{equation*}
Setting $\bar y:= y(\bar u)$, we define $\mathcal{R}:=\{y(u):u\in\U\}$, and let $\mathbf{R}, \mathbf{K}$ be Boltyanski approximating cones at $\bar y$ to $\mathcal{R}$ and $\mathcal{S}$, respectively.
Then, there exists $(\lambda, \lambda_c) \in \left( \R^{n+1} \right)^\star$ such that
\begin{equation*}
\begin{split}
    &(a) \quad (\lambda, \lambda_c) \neq (0,0),\\
    &(b) \quad \lambda \in -\mathbf{K}^\perp \quad \mbox{and } \lambda_c \leq 0, \\
    &(c) \quad
    \max_{v \in \mathbf{R}} \left( \min_{p \in \partial \Psi(\bar y) }\langle \lambda + \lambda_c p,v\rangle \right) = 0.
\end{split}
\end{equation*}
\end{theorem}
\begin{proof}
Following the proof of \cite[Theorem~5.1]{book_rampazzo}, let us introduce the sets $\tilde{\mathcal{S}} := \{ (x, \eta) \in \R^{n+1}: x \in \mathcal{S}, \eta < \Psi(\bar y) \} \cup \{(\bar y, \Psi(\bar y))\}$ and $\tilde{\mathcal{R}} := \text{epi}(\Psi_{|\mathcal{R}})$,
and their respective Boltyanski approximating cones at $(\bar y, \Psi(\bar y))$, i.e., 
$\tilde{\mathbf{K}} = \{ (w,\nu): w \in \mathbf{K}, \nu \geq 0 \}$ and, owing to Proposition~\ref{prop:cone_epi}, $\tilde{\mathbf{R}} := \{ \big(v, \sigma_{\partial \Psi(\bar y)}(v) + \eta \big): v \in \mathbf{R}, \eta \geq 0 \}$.
In virtue of \cite[Corollary~4.1]{book_rampazzo}, since $\tilde{\mathcal{R}}$ and $\tilde{\mathcal{S}}$ are locally separate (see \cite[Lemma~5.1]{book_rampazzo}), it follows that $\tilde{\mathbf{R}}$ and $\tilde{\mathbf{K}}$ are linearly separable, i.e., there exists $ (\lambda, \lambda_c) \in \left( \R^{n+1}\right)^\star$ s.t. $(\lambda,\lambda_c) \neq (0,0)$,
\begin{align*}
     \langle (\lambda, \lambda_c), \tilde{w} \rangle \geq 0 &\qquad \forall \tilde{w} \in \tilde{\mathbf{K}}, \\
     \langle (\lambda, \lambda_c), \tilde{v} \rangle \leq 0  &\qquad \forall \tilde{v} \in \tilde{\mathbf{R}}.
\end{align*}
From the first inequality, we deduce point (b) of the thesis, while the second inequality yields 
\begin{equation*}
    \langle \lambda, v \rangle + \lambda_c (\sigma_{\partial \Psi(\bar y)}(v) + \eta) \leq 0
\end{equation*}
for every $v\in\mathbf{R}$ and $\eta\geq 0$,
which implies for $\eta =0$ that $\langle \lambda, v \rangle + \lambda_c \sigma_{\partial \Psi(\bar y)}(v) \leq 0$ for every $v\in\mathbf{R}$. Since the equality is attained for $v=0$, we deduce that
\begin{equation*}
    \max_{v \in \mathbf{R}} \left( \langle \lambda, v \rangle + \lambda_c \sigma_{\partial \Psi(\bar y)}(v) \right) = 0.
\end{equation*}
Finally, recalling that $\lambda_c \leq 0$, the thesis follows from \eqref{eq:def_support}. 
\end{proof}

\subsection{Proof of Theorem~\ref{thm:pmpfinite}}
We finally prove the necessary optimality conditions associated to \eqref{eq:def_funct} using Theorem~\ref{thm:abstr_PMP}.
In our framework, $\U$ is the space of admissible controls, $\R^n=(\R^{d\times N})^M$, and $y:\U\to\R^n$ maps every control $u\in \U$ to the final-time state of the corresponding solution of \eqref{eq:ext_node}.
Since we do not deal with final-time constraints, we have that $\mathcal{S}=(\R^{d\times N})^M$.
Moreover, we set $\mathcal{R}:=
y(\U)
$.
As a direct application of \cite[Corollary~6.1]{book_rampazzo}, we obtain the approximating cones to $\mathcal{R}$. 

\begin{prop} \label{prop:appr_reachable}
Let $(\bar u, \bar X)$ be a strong local minimizer for \eqref{eq:def_funct}, and for $r\geq 1$ let $\{t_1, \ldots, t_r\} \subseteq [0,T]$ be arbitrary distinct Lebesgue points for $t\mapsto \mathcal{F}(\bar X(t), \bar u(t))$, and let us take $\omega_1, \ldots, \omega_r \in U$. Then, if we set $W_k := \mathcal{F}(\bar X(t_k), \omega_k) - \mathcal{F}(\bar X(t_k), \bar u (t_k))$,  
\begin{equation} \label{eq:bol_cone_tot}
    \mathbf{R} := \mathrm{span}^+_{k=1,\ldots,r} \left\{ \mathcal{M}(t_k, T) 
    W_k \right\}
\end{equation}
is a Boltyanski approximating cone to $\mathcal{R}$ at $\bar X(T)$, where 
$\mathcal{M}(\cdot, \cdot)$ is the fundamental matrix\footnote{See, e.g., \cite[Section~2.2]{book_bressan}.} associated to the linearized equation
$\dot{V}(t) = \nabla_{\!X} \mathcal{F}\left( \bar X(t),\bar u (t) \right)  V(t)$.
\end{prop}

We observe that the Boltyanski cone \eqref{eq:bol_cone_tot} reflects the particular structure of the dynamics \eqref{eq:ext_node}. Indeed, we can write
$    \mathbf{R}= \prod_{i=1}^M \left( \prod_{j=1}^N 
    \mathrm{span}^+_{k=1,\ldots,r} \left\{ M^{i,j}(t_k, T) 
    w_k^{i,j} \right\}
    \right)$,
where $w_k^{i,j}:=F(\bar x_i^j(t_k), \omega_k) - 
{F}(\bar x_i^j(t_k), \bar u(t_k))$, and  $M^{i,j}(\cdot,\cdot)$ is the fundamental matrix associated to the linearization of the trajectory $\bar x_i^j$, i.e., $\dot v(t) = \nabla_x F(\bar x_i^j(t), \bar u(t))v(t)$.
We are now in position to present the proof of Theorem~\ref{thm:pmpfinite}.

\begin{proof}[Proof of Theorem~\ref{thm:pmpfinite}] 
In virtue of Theorem~\ref{thm:abstr_PMP}, since $\mathcal{S}=(\R^{d\times N})^M$, we have that $\mathbf{K}=(\R^{d\times N})^M$ and $\mathbf{K}^\perp = \{0\}$. 
Moreover, using the same notations as in Section~\ref{sec:model}, we have that $\Psi(X) = \frac1M \sum_{i=1}^M\tilde g_i(X_i)$.
Hence, by setting $\lambda_c=-1$, we consider the approximating cone $\mathbf{R}$ to $\mathcal{R}$ at $\bar X(T)$ constructed in Proposition~\ref{prop:appr_reachable} with Lebesgue points $\{ t_1,\ldots,t_r\}$, and from Theorem~\ref{thm:abstr_PMP} it follows that
\begin{equation*}
    0
    = 
  \max_{V \in\mathbf{R} \cap \overline{B_1(0)}} \min_{P \in - \partial \Psi(\bar X(T))}
  \langle P, V \rangle,
\end{equation*}
since the maximum is attained at $0\in \mathbf{R} \cap \overline{B_1(0)}$. 
Moreover, by using von Neumann's Minimax Theorem (see, e.g., \cite[Theorem~3.4.6]{book_vinter}) we obtain that
\begin{equation*}
    \min_{P \in - \partial \Psi(\bar X(T))}
    \max_{V \in \mathbf{R} \cap \overline{B_1(0)}} 
    \langle P, V \rangle = 0,
\end{equation*}
i.e., there exists $P^* \in - \partial \Psi(\bar X(T))$ such that $\langle P^*, V \rangle \leq 0$  for every $V \in \mathbf{R} \cap \overline{B_1(0)}$. Since the last inequality is invariant by positive rescaling, we deduce that 
\begin{equation*}
    \langle P^*, \mathcal{M}(t_k, T) \left( 
    \mathcal{F}(\bar X(t_k), \omega_k)- 
    \mathcal{F}(\bar X(t_k), \bar u(t_k)\right) \rangle \leq 0.
\end{equation*}
From the structure of the cost function $\Psi$ and Lemma~\ref{lem:max_rule}, 
it follows that the covector $P^*$$= (P^*_1,\ldots,P^*_M)$ has the form $P^*_i = -\frac1M\left(\gamma_i^1 \nabla_x g_i\big(\bar x_i^1(T)\big), \ldots, \gamma_i^N \nabla_x g_i\big(\bar x_i^N(T)\big)\right)$.
Then, the thesis follows
from a classical infinite intersection argument, 
as detailed e.g. in the proof of \cite[Theorem~5.7.1]{book_sussmann}.
\end{proof}

\section{NUMERICS} \label{sec:numerics}
In this section, we present numerical experiments designed to validate our interpretation of the PMP in the context of robust neural network training.\\
Here, we approach the robust training of deep neural networks from a control perspective. Specifically, we interpret neural networks as discretized representations of controlled ODEs such as the one depicted in \eqref{eq:nODE}. Then, the network parameter training is viewed as a control problem. Our objective is to minimize the cost function \eqref{eq:def_funct}, which is defined in order to obtain a network which robustly classifies the input data. In particular, we employ the PMP outlined in Theorem \ref{thm:pmpfinite} and derived in Section \ref{sec:proof}, to identify the controls/network parameters that yield effective and robust classification performance. 
The utilization of the PMP for neural network training has already been successfully applied and explored in \cite{LCCE19,BCFH_23}. In practice, this approach typically involves employing a shooting method, such as the one outlined in Algorithm \ref{alg:shooting}, which consists of repeating the forward evolution of the trajectories, the backward evolution of the adjoint variables, and the update of the controls, until convergence. For further insights into this methodology, we direct readers to \cite{Sakawa_80}. 
It is crucial to emphasize that the novelty of our approach lies in the selection of the weights to be incorporated into \eqref{eq:funct_smooth}, rather than in the use of the PMP for training neural networks.
To keep the analysis simple, we focus on a classification task in $2$d, but we defer the extension to higher-dimensional experiments to future research. The primary objective here is to robustly train a network to distinguish between data points separated by a nonlinear boundary.  
\begin{algorithm}[ht!]
\scriptsize
\caption{Shooting method}\label{alg:shooting}
\KwData{\\ $u^0:$ initial guess for controls; $\mathrm{iter\_max}:$ number of shooting iterations\; $\Delta t:$ time-discretization of the interval $[0,T]$;
$\tau:$ memory parameter\;
}
\KwResult{$u^{\mathrm{iter\_max}}$}
$\text{time\_nodes} \gets T / \Delta t$\;
\For{$k = 0, \ldots, \mathrm{iter\_max}$}{
\For{$i=1,\ldots,M$, $j=1\ldots,N$}{
\For{$n = 1, \ldots, \mathrm{time\_nodes}$}{
    $x_i^j(t_{n+1}) \gets x_i^j(t_n) + \Delta t\, {F}(x_i^j(t_n), u^k(t_n))$\; 
}
}
\For{$i=1,\ldots,M$, $j=1,\ldots,N$}{
$p_i^j(t_{\mathrm{time\_nodes}}) \gets 
-\frac{\gamma_i^j}{M}\nabla_{\!x} g_i \left(x_i^j(t_{\mathrm{time\_nodes}})\right)$\;
\For{$n = \mathrm{time\_nodes}, \ldots, 1$}{
    $p_i^j(t_{n-1}) \gets p_i^j(t_{n}) + \Delta t\, p_i^j(t_{n})\cdot \nabla_{\!x}{F}\big(x_i^j(t_{n}), u^k(t_{n})\big)$\; 
}
}
\For{$n = \mathrm{time\_nodes}, \ldots, 1$}{
    $X(t_{n-1})\gets(x_i^j)_{i=1,\ldots,M}^{j=1,\ldots,N}$ and $P(t_{n-1})\gets(p_i^j)_{i=1,\ldots,M}^{j=1,\ldots,N}$\;
    $u^{k+1}(t_{n-1}) \gets \max_{\omega \in U} \left[ H(X(t_{n-1}), P(t_{n-1}), \omega)\right.$ \\
    \qquad\qquad\qquad\qquad\qquad\qquad\qquad
    $\left.- \frac{1}{2\tau} \| \omega - u^k(t_{n-1})\|_2^2 \right]$\;
}
If needed, update $(\gamma_i^j)_{i=1,\ldots,M}^{j=1,\ldots,N}$\;
}
\end{algorithm}

\subsection{The Classification task}
\textbf{The training data:} We generate two classes by uniformly sampling $M = 200$ data points in the domain $[0,1]^2$ excluding a separation region (shaded yellow in Fig.~\ref{fig:comparison_plots}). These data points are labeled according to their positions ($y=1=$ `above' or $y=-1=$ `below') w.r.t.\ a predetermined nonlinear boundary (depicted as a yellow line in Fig.~\ref{fig:comparison_plots}).
The variable $y$ introduced in \eqref{eq:roboptim} denotes the label associated with each data point. Since this variable is incorporated in the cost $g_i$ for every datum $i$, it does not explicitly appear in Algorithm~\ref{alg:shooting}.

\vspace{5pt}

\textbf{The network:} We employ a residual neural network of $20$ layers, which corresponds to a discretization of the interval $[0,T]$ with $T=1$ and $\Delta t= 0.05$. The network is interpreted as an explicit Euler discretization of the $2d$-controlled ODE
\begin{equation} \label{eq:nODE_exp}
\begin{cases}
    \dot{x}(t) = F(x(t),u(t))= \sigma\big(W(t)x(t) + b(t)\big),\\
    x(0)= \xi^0.
\end{cases}
\end{equation}
Here, $\sigma$ denotes an activation function acting component-wise, set to be the hyperbolic tangent in our experiments. The state-space of this system is $\R^2$, while the control variable is $u(t)=(W(t),b(t))\in\R^{2\times2}\times\R^2$. Here, given an admissible control $t\mapsto (W(t),b(t))$, evaluating the neural ODE written above on our dataset means solving \eqref{eq:nODE_exp} with $\xi^0=x_i^0 + \alpha_i^j$, with $i=1,\ldots,M$ and  $j=1,\ldots,N$.

\vspace{5pt}

\textbf{The adversarial budget:} We fix $N = 4$ equidistant perturbations with adversarial budget $\epsilon = 0.02$ around each training point. This budget is chosen sensibly with the class separation region (yellow in Fig.~\ref{fig:comparison_plots}) in mind: The perturbations may populate the region but an unwanted overlap of opposite labels is prevented. We remark that in high-dimensional implementations, such as image classification, the attack's budget is typically assumed to be very small compared to the class distance, so that the labels 
are unaffected.

\vspace{5pt}

\textbf{The cost function:}
As loss functions $g_i$, we use a cost that promotes separation and clustering of the two classes. In ML, it is typical to append a trainable linear layer at the end of the network to project the classifier's output, and then to employ a penalization on the mismatch. However, in order to maintain an optimal control formulation and for more accurate comparison of the weight choices below, we prefer to incorporate a fixed projection within the loss function. Consequently, we use the following cost:
\begin{equation*}
g_i(x_i^{j}(T))  = 
    \begin{cases}
        e^{v\cdot x_i^{j}(T)} 
        & y_i=1,\\
        e^{-v\cdot x_i^{j}(T)} 
        &   y_i=-1,\\
    \end{cases}
\end{equation*}
where $v=(1,-1)$. It is important to note that while our method is applicable to any type of cost function, including quadratic or more complex formulations, we have chosen the above expression for visualization purposes.

\vspace{5pt}

\textbf{The training:} In Algorithm~\ref{alg:shooting}, we provide the shooting method used for the numerical solution of the PMP, and hence for training the network. The connection between the shooting method and the training of neural networks with gradient descent has been investigated in  \cite{LCCE19}, \cite{Benning_dl_as_oc}. Notice that the extra term in the Hamiltonian (at line $17$) comes from \cite{Sakawa_80} for stability reasons. We set the initial guess $u^0$ by randomly generating standard Gaussian values for each entry and for each time-node, while the number of shooting iterations is $\mathrm{iter\_max} = 1000$. It is worth noting that the maximization step in line $16$ of Algorithm \ref{alg:shooting} can be performed in various ways, such as gradient ascent. However, we opt for the approach used in \cite{CFS_23}, which employs a fixed-point method to find the point where the gradient of the augmented Hamiltonian vanishes.

\subsection{The choice of the weights}
One crucial challenge is that the PMP derived in Theorem~\ref{thm:pmpfinite} does not specify how to choose the weights $\gamma^j_i$ at line $19$ of Algorithm~\ref{alg:shooting}. Therefore, we compare three  options.

\vspace{5pt}

\textbf{Uniform robust:} The first approach involves assigning uniform weights to all perturbations. In this method, we consider the perturbations as training data points and minimize the loss based on them. However, this straightforward approach is not connected to a minimax problem and does not align with the theoretical framework, where the weights play a specific role in emphasizing the perturbations that achieve the maximum.

\vspace{5pt}

\textbf{Weighted robust:} The second option is to introduce the following weight functions for the perturbations:
\begin{equation}\label{eq:weights}
    \gamma_i^j= \frac{1}{C_i}e^{c g_i(x^j_i(T)) } \quad \text{with } C_i= \sum_{j=1}^N e^{c g_i(x^j_i(T))}, 
\end{equation}
where $c >0$ is a constant chosen a priori.
This choice is inspired by Laplace's principle \cite{Laplace2} and it is reminiscent of Gibbs measures. 
These weights must be updated each time the control is adjusted, and they are determined based on the cost $g_i$ of the network's output $x_i^j(T)$, which we aim to maximize across perturbations. Subsequently, these weights are normalized according to the cost achieved by the other perturbations. The benefit of this approach is that if two perturbations (almost) achieve the maximum, they will have (almost) equal weights, leading to a more stable method. This peculiarity stands in contrast to the instability of the worst-case approach described below. The constant $c$ needs to be chosen appropriately. Setting it to zero (or too small) is equivalent to the uniform weighting method described above. Similar considerations have already been done in \cite{Li_tilted}, motivated by the relaxation of the max function by the LogSumExp.

\vspace{5pt}

\textbf{Worst-case robust:} 
The third approach consists of explicitly computing the worst-case scenario at each step --- which involves identifying for every datum the perturbation that achieves the maximum --- and in updating the controls using exclusively these data-points. During the execution of the shooting method, the worst-case perturbations are typically different at each iteration. As a matter of fact, the cost is highly oscillating along the iterative steps, since improving the performance on the worst-cases might result in deteriorating the behavior of the others. At each iteration of the shooting method, this approach is equivalent to set the coefficients $(\gamma_i^j)_{i=1,\ldots,M}^{j=1,\ldots,N}$ according to \eqref{eq:weights}, and then to send $c \to \infty$.

\subsection{Results}
In Fig.~\ref{fig:comparison_plots} we display the classification prediction probabilities of the trained models. 
We compare all the four methods described above, namely the non-robust approach using the unperturbed set of training data, the uniform robust method, the weighted robust approach with weights as \eqref{eq:weights} with $c=100$ and finally the worst-case robust method.
\begin{figure}[b]
\includegraphics[trim=120 120 120 100, clip, scale=0.26]{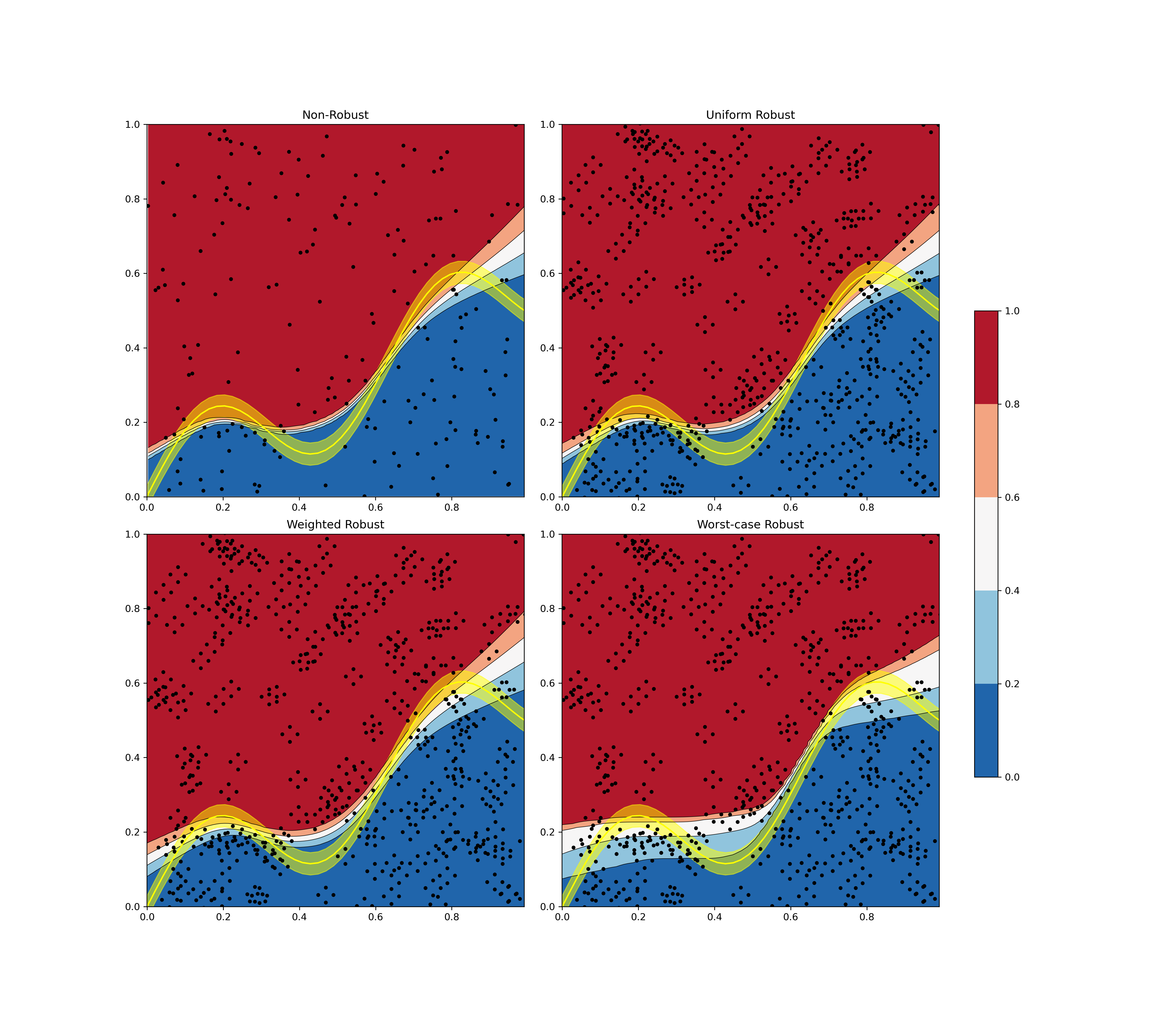}
\caption{Classification level-sets on $[0,1]^2$: the color bar indicates the confidence of prediction of one class (red above the yellow margin) or the other class (blue below the yellow margin).}\label{fig:comparison_plots}
\end{figure}
For the robust methods, as training input, we use the perturbed version of the data employed in the non-robust method, i.e. $M \times N = 800$ training points, with unchanged parameter initialization. Upon initial visual analysis, it is apparent that while the robust methods may not significantly outperform the non-robust method in precisely identifying the decision boundary, they excel in recognizing areas with higher risk of misclassification and making more cautious predictions. To provide a more quantitative comparison of all methods, we present in Table~\ref{table:comparison} several metrics obtained by averaging over multiple experiments. These metrics are: test accuracy (on $[0,1]^2$), classification margin accuracy of the excluded margin of unperturbed training points, and high-confidence mistakes, which determine the share of all misclassifications in $[0,1]^2$ with prediction confidence of more than $70\%$. This last metric highlights the algorithm's capability to identify areas of ambiguity.
\begin{table}[t]
\scriptsize
\begin{tabular}{l c c c}
\\
\hline
\, & \textbf{Test Accuracy} & \textbf{Margin Accuracy} & \textbf{High-confidence}\\
& \, & \, & \textbf{ Mistakes}\\
\hline
Non-robust & 94,12\,\% & 56,52\,\% & 78,50\,\%\\
\hline
Uniform robust & 94,26\,\% & 56,80\,\% & 72,57\,\%\\
\hline
Weighted robust & 94,12\,\% & 57,14\,\% & 65,13\,\% \\
\hline
Worst-case robust & 94,68\,\% & 58,31\,\% & 55,76\,\%\\
\hline
\end{tabular}
\caption{\vspace{-0.7cm}}\label{table:comparison}
\end{table}
Finally, we provide a comparison of robustness in Fig.~\ref{fig:semilog_loss}, where we examine the value of $J$ as defined in \eqref{eq:def_funct}. To achieve robustness, the objective is to minimize this value, motivated by the minimax formulation \eqref{eq:roboptim}. The key observation here is that the worst-case robust method achieves the lowest loss, which aligns with the results presented in Table \ref{table:comparison}. However, this method exhibits considerable oscillations due to the instability of the maxima. Indeed, when handling scenarios involving multiple perturbations that are close to the maximum, the worst-case method would select a single perturbation and adjust the parameters accordingly. This approach may favor one direction while neglecting the others. In contrast, the weighted robust method considers the influence of all perturbations. This is confirmed by the plot in Fig.~\ref{fig:semilog_loss}: our proposed weighted robust approach generally exhibits slightly reduced robustness compared to the worst-case robust method, but it offers more  stability during training.
\begin{figure}[b]
\centering
\includegraphics[scale=0.35, trim = {0 0 0 1cm}, clip]{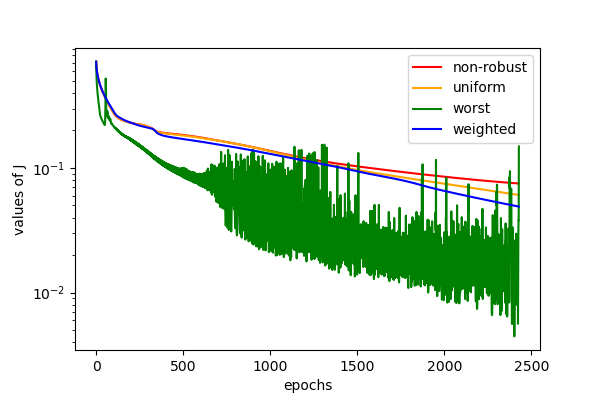}
\caption{Robustness measure displayed on a semilogarithmic scale.}\label{fig:semilog_loss}
\end{figure}

\section{CONCLUSIONS}
In this work, we investigate the minimax optimal control problem of neural ODEs in order to improve adversarially robust training of deep neural networks. In the continuous ODE setting, we provide a proof of the first order optimality conditions in form of the PMP. 
Inspired by these conditions, we present a numerical scheme for training neural networks that incorporates weighted adversarial attacks for each training point. This method achieves promising results in low dimensions: it surpasses the uniform robust method in terms of accuracy and exhibits greater stability in terms of training behavior compared to the worst-case approach. 
Our novel approach has potential for higher dimensional settings, where the precise computation of the local maximum is infeasible, but a weighted approximation could improve robustness.





\section*{ACKNOWLEDGMENT}

C.C.\ is partially supported by 
the DFG SPP 2298. A.S.\ is partially supported by INdAM-GNAMPA. T.W.\ is supported by the Austrian Science Fund (FWF) grant no.~J 4681-N.
Finally, the authors want to thank Prof. Franco Rampazzo.


\begin{thebibliography}{99}
\bibitem{hassabis2021}
J.~Jumper, R.~Evans, A.~Pritzel, \emph{et~al.}
  Highly accurate protein structure prediction with
  alphafold. 
  In: Nature, 596:583--589, 2021.


\bibitem{resnet2016}
K.~He, X.~Zhang, S.~Ren, and J.~Sun. 
Deep residual learning for image recognition. 
In: {2016 IEEE CVPR}, 770--778, 2016.

\bibitem{Good2016}
I.~Goodfellow, Y.~Bengio, and A.~Courville, \emph{Deep Learning}.\hskip 1em
  plus 0.5em minus 0.4em\relax MIT Press, 2016, http://www.deeplearningbook.org.

\bibitem{szegedy2013}
C.~Szegedy, W.~Zaremba, I.~Sutskever, J.~Bruna, D.~Erhan, I.~Goodfellow, R.~Fergus. 
Intriguing properties of neural networks. 
In: 2nd ICLR, 2014.

\bibitem{Good2015}
I.~J. Goodfellow, J.~Shlens, and C.~Szegedy. 
Explaining and Harnessing Adversarial Examples. 
In: 3rd ICLR, 2015. 

\bibitem{weinan2017}
W. E.
A proposal on machine learning via dynamical systems.
In: Comm. Math. Stat., 1(5):1--11, 2017.

\bibitem{haber2017}
E.~Haber, L.~Ruthotto.
Stable architectures for deep neural networks.
In: Inverse problems, 34(1), 2018.

\bibitem{CRBD18}
R.T.Q. Chen, Y. Rubanova, J. Bettencourt, D.K. Duvenaud.
{Neural Ordinary Differential Equations}. 
In: Advances in Neural Information Processing Systems, 2018.


\bibitem{dupont2019}
E.~Dupont, A.~Doucet, Y.~W. Teh.
Augmented neural odes. 
{Advances in Neural Information Processing Systems}, 2019.

\bibitem{pont62}
L.~S. Pontryagin, V.~G. Boltyanskii, R.~V. Gamkrelidze, E.~F. Mishchenko.
{The mathematical theory of optimal processes}.
In: Interscience Publishers John Wiley \& Sons Inc.,
1962.


\bibitem{LCCE19}
Q.~Li, L.~Chen, C.~Tai, W. E.
Maximum principle based algorithms for deep learning.
In: J. Mach. Learn. Res.,  18(1):5998–-6026, 2017.

\bibitem{doubletrillos}
C.~G. Trillos, N.~G. Trillos.
On the regularized risk of distributionally robust learning over deep neural networks.
In: Res. Math. Sci., 9(54), 2022.

\bibitem{CFS_23}
C. Cipriani, M. Fornasier, A. Scagliotti.
From NeurODEs to \\ AutoencODEs: a mean-field control framework for width-varying neural networks. 
In: Eur. J. Appl. Math., 2024.

\bibitem{EGPZ}
C.~Esteve, B.~Geshkovski, D.~Pighin, E.~Zuazua.
{Large-time asymptotics in deep learning}. 
In: arXiv:2008.02491

\bibitem{ruiz2023neural}
D.~Ruiz-Balet, E.~Zuazua. 
Neural ode control for classification, approximation, and transport.
In: {SIAM Review}, 65(3):735--773, 2023.

\bibitem{madry2018}
A.~Madry, A.~Makelov, L.~Schmidt, D.~Tsipras,  A.~Vladu. Towards deep learning models resistant to adversarial attacks.
In: 6th ICLR, 2018.

\bibitem{shaham18}
U.~Shaham, Y.~Yamada, S.~Negahban. 
Understanding adversarial training: Increasing local stability of supervised models through robust optimization.
In: {Neurocomputing},  307:195--204, 2018.

\bibitem{wald45}
A.~Wald. 
Statistical decision functions which minimize the maximum risk.
In:  {Annals of Mathematics}, 46(2):265--280, 1945.

\bibitem{tramer2017transfer}
F.~Tram{\`e}r, N.~Papernot, I.~Goodfellow, D.~Boneh, P.~McDaniel.
The space of transferable adversarial examples.
In: arXiv:1704.03453, 2017.

\bibitem{frank2022existence}
N.~S. Frank, J.~Niles-Weed.
Existence and minimax theorems for adversarial surrogate risks in binary classification. 
In: arXiv:2206.09098, 2022.


\bibitem{article_vinter} 
R.B. Vinter. 
Minimax optimal control. 
In: SIAM J. Control Optim., 4(3):939-968, 2005.


\bibitem{BCFH_23}
B. Bonnet, C. Cipriani, M. Fornasier, H. Huang.
A measure theoretical approach to the mean-field maximum principle for training NeurODEs. 
In: Nonlinear Analysis, 227:113-161, 2023.

\bibitem{book_vinter} 
R.B. Vinter. Optimal control.
In: Birkh\"auser, 1(1), 2000.


\bibitem{book_rampazzo} 
M. Motta, F. Rampazzo. 
An Abstract Maximum Principle for constrained minimum problems. 
In: arXiv:2310.09845, 2023.


\bibitem{book_convex} 
J.-B. Hiriart-Urruty, C. Lemaréchal.
Fundamentals of Convex Analysis.
In: Springer Science \& Business Media, 2004.

\bibitem{book_bressan}
A. Bressan, B. Piccoli. 
Introduction to the mathematical theory of control. 
In: AIMS, Springfield, 2004. 

\bibitem{book_sussmann}
H.J. Sussmann. 
Geometry and Optimal Control. 
In: Baillieul, J., Willems, J.C. (eds) Mathematical Control Theory. Springer, 1999.



\bibitem{Sakawa_80}
Y. Sakawa, Y. Shindo.
On global convergence of an algorithm for optimal control. 
In: IEEE Trans. Automat. Contr. 25(6):1149-1153, 1980.

\bibitem{Benning_dl_as_oc}
M. Benning, E. Celledoni, M. Ehrhardt, B. Owren, C. B. Schhönlieb.
Deep learning as optimal control problems: models and numerical methods. 
In: Journal of Computational Dynamics, 2019.

\bibitem{Laplace2} 
A. Dembo and O. Zeitouni.
Large deviations techniques and applications.
In: Springer Science \& Business Media, 38, 2009.

\bibitem{Li_tilted}
T. Li, A. Beirami, M. Sanjabi, V. Smith. 
Tilted Empirical Risk Minimization. 
In: 8th ICLR, 2020.


\end{thebibliography}
\end{document}